\documentclass[12pt]{amsart}

\usepackage[DIV=14,BCOR=1.5cm]{typearea}

\usepackage[colorlinks=true,hyperindex,linktocpage=true]{hyperref}
\usepackage{theoremref}
\usepackage{mathtools}
\usepackage{todonotes}
\usepackage{enumitem}
\setlist{nolistsep}
\usepackage{amssymb}
\setuptodonotes{size=\tiny}

\makeatletter
\newcommand*\defbb[1]{
	\expandafter\newcommand\csname I#1\endcsname{\mathbb{#1}}}
\newcommand*\defbbs[1]{
	\@for\@i:=#1\do{\expandafter\defbb\expandafter{\@i}}}
\makeatother
\defbbs{A,B,C,D,E,F,G,H,I,K,L,M,N,O,P,Q,R,S,T,U,V,W,X,Y,Z} 

\makeatletter
\newcommand*\deffrak[1]{
	\expandafter\newcommand\csname frak#1\endcsname{\mathfrak{#1}}}
\newcommand*\deffraks[1]{
	\@for\@i:=#1\do{\expandafter\deffrak\expandafter{\@i}}}
\makeatother
\deffraks{a,b,c,d,e,f,g,h,i,j,k,l,m,n,o,p,q,r,s,t,u,v,w,x,y,z}

\makeatletter
\newcommand*\defcal[1]{
	\expandafter\newcommand\csname cal#1\endcsname{\mathcal{#1}}}
\newcommand*\defcals[1]{
	\@for\@i:=#1\do{\expandafter\defcal\expandafter{\@i}}}
\makeatother
\defcals{A,B,C,D,E,F,G,H,I,J,K,L,M,N,O,P,Q,R,S,T,U,V,W,X,Y,Z}

\makeatletter
\newcommand*\defopname[1]{
	\expandafter\newcommand\csname #1\endcsname{\operatorname{#1}}}
\newcommand*\defopnames[1]{
	\@for\@i:=#1\do{\expandafter\defopname\expandafter{\@i}}}
\makeatother
\defopnames{M,adj,tr,SO,SL,ord,Div,Spin,Mp,reg,Types,BP,type}
\DeclareMathOperator{\di}{Div}
\DeclareMathOperator{\spbdiv}{\operatorname{Div}_\partial^{\operatorname{sp}}}

\newcommand{\bs}{\ensuremath{\backslash}}

\theoremstyle{definition}
\newtheorem{defn}{Definition}[section]

\theoremstyle{plain}
\newtheorem{thm}[defn]{Theorem}
\newtheorem{lem}[defn]{Lemma}
\newtheorem{cor}[defn]{Corollary}
\newtheorem{prop}[defn]{Proposition}

\theoremstyle{remark}
\newtheorem{rem}[defn]{Remark}

\numberwithin{equation}{section}
\title{A converse theorem for Borcherds products in signature $(2,2)$}
\author{Patrick Bieker}
\address{
Department of Mathematics, Massachusetts Institute of Technology, 77 Massachusetts Ave, Cambridge, MA
02139
}
\email{patri323@mit.edu}
\author{Paul Kiefer}
\address{
Fakultät für Mathematik, Universität Bielefeld, D-33501 Bielefeld, Germany.
}
\email{paul.kiefer@uni-bielefeld.de}

\thanks{The authors are partially funded by the Deutsche Forschungsgemeinschaft (DFG, German Research Foundation) -- SFB-TRR 358/1 2023 -- 491392403. P.B. is supported by the Deutsche Forschungsgemeinschaft, project number 520675682.}

\begin{document}

\begin{abstract}
    We show that a modular unit on two copies of the upper half-plane is a Borcherds product if and only if its boundary divisor is a special boundary divisor. Therefore, we define a subspace of the space of invariant vectors for the Weil representation which maps surjectively onto the space of modular units that are Borcherds products. Moreover, we show that every boundary divisor of a Borcherds product can be obtained in this way. As a byproduct we obtain new identities of eta products.
\end{abstract}

\maketitle

{\parskip 0pt 
\tableofcontents}

	\section{Introduction}

    In his famous paper \cite{Borcherds}, Borcherds constructed a multiplicative map from the space of weakly holomorphic modular forms of weight $1 - \frac{n}{2}$ for the Weil representation associated to an even lattice $L$ of signature $(2, n)$ to meromorphic orthogonal modular forms for a certain subgroup $\Gamma(L)$ of the orthogonal group of $L$. These orthogonal modular forms, also called Borcherds products, have infinite product expansions, from which one can see that their divisors are supported on so called special divisors. Borcherds posed the question, when a meromorphic orthogonal modular form is a Borcherds product, see \cite[Problem 16.10]{Borcherds}.

    A first step into this direction was done in \cite{BruinierHabil}, where it was shown that if $n \geq 3$ and $L$ splits two hyperbolic planes over $\IZ$, then every meromorphic orthogonal modular form whose divisor is supported on special divisors is a Borcherds product. The result has been generalized to the case that $L$ splits a hyperbolic plane and a scaled hyperbolic plane over $\IZ$, see \cite{BruinierConverse}. Both results construct a Borcherds product that has the same divisor as the meromorphic orthogonal form using the weak converse theorem \cite[Theorem 4.23]{BruinierHabil} and then use the Koecher principle to show that they must be equal up to a constant factor. For lattices of signature $(2,1)$ or $(2,2)$ and Witt rank 2, neither the weak converse theorem nor the Koecher principle are available. The case $n = 1$ was considered in \cite{BruinierSchwagenscheidtConverse}, where a converse theorem using generalized Borcherds products was proven. The aim of the present paper is to tackle the problem for lattices of signature $(2,2)$ with Witt rank $2$. We will define a space of \emph{special boundary divisors} and show the following converse theorem.
    \begin{thm}[{see \thref{cor:ModularUnitsThatAreBorcherdsProducts}}]
        Let $F : \IH \times \IH \to \IC$ be a modular form of some weight and some character of finite order whose divisor is concentrated on the boundary, i.e. it is a modular unit. Then $F$ is a Borcherds product if and only if it has a special boundary divisor.
    \end{thm}
    We will now describe our results in more detail. Let $L = U(N') \oplus U(N)$ be the sum of two scaled hyperbolic planes and let $V = L \otimes \IQ$. 
    Then $L$ has signature $(2, 2)$ and Witt rank $2$. 
    For an easier exposition we will assume that $N = p^r, N' = p^{r'}$ are prime powers in the introduction. 
    Let $\rho_L$ be the Weil representation of the metaplectic group $\Mp_2(\IZ)$ on $\IC[L' / L]$. 
    The hermitian symmetric domain associated to the orthogonal group $O(V)$ can be identified with $\IH \times \IH$. Under this identification the discriminant kernel $\Gamma(L) = \ker(O(L) \to O(L' / L))$ is identified with a certain discrete subgroup of $\SL_2(\IZ) \times \SL_2(\IZ)$ in this setting. The Borcherds lift becomes a multiplicative map from weakly holomorphic modular forms $f : \IH \to \IC[L' / L]$ of weight $0$ with respect to the Weil representation $\rho_L$ to meromorphic modular forms $F : \IH \times \IH \to \IC$ with respect to $\Gamma(L)$. An important subspace of the space of weakly holomorphic modular forms of weight $0$ is the space of invariant vectors $\IC[L' / L]^{\Mp_2(\IZ)}$ which has already been studied in \cite{EhlenSkoruppa, Bieker, Zemel2021, ManuelInvariants}. Under the Borcherds lift, these map to modular units, i.e. holomorphic modular forms whose divisor is concentrated on the boundary 
    of the Bailey-Borel compactification $X(L)$ of the quotient $Y(L) = \Gamma(L) \setminus \IH \times \IH$.
    In fact the boundary divisor of a Borcherds product is completely determined by an invariant vector, see \thref{lem:boundary-mult-weyl-vect}. In particular, to study the boundary divisors of Borcherds lifts, it is sufficient to study the Borcherds lifts of invariant vectors.

    In our setting the boundary of the Bailey-Borel compactification can be described explicitly as follows.
    The boundary of $\IH \times \IH$ consists of $0$-dimensional cusps given by elements in $\IP^1(\IQ) \times \IP^1(\IQ)$ and $1$-dimensional cusps isomorphic to $\{a / c\} \times \IH,\ \IH \times \{a / c\}$ for $a / c \in \IP^1(\IQ)$ where $a,c$ are integers with $\gcd(a, c) = 1$ and $c \geq 0$. These $1$-dimensional boundary components are in $1$-$1$-correspondence with rational $2$-dimensional isotropic subspaces $I_{a, c}^1, I_{a, c}^2 \subseteq V$. We show that the image of $I_{a, c}^* \cap L', *=1,2$ in $L' / L$ is a self-dual isotropic subgroup $H_{a, c}^* = \type_L(I_{a, c})$ and call this the type of $I_{a, c}^*$ following \cite{DriscollSpittlerScheithauerWilhelm}. This yields a subspace $\IC[L' / L]_{\Types(L)} \subseteq \IC[L' / L]^{\Mp_2(\IZ)}$ generated by the characteristic functions of the types $H_{a, c}^*$ of rational $2$-dimensional isotropic subspaces $I_{a, c}^*$. We obtain the following

    \begin{prop}[{see \thref{lem:space-types-relation} and \thref{lem:space-types-dimension}}]
        We have
            \begin{align*}
                 & \sum_{a \in p \IZ/p^{r'} \IZ} v^{H^1_{a,1}}  +\sum_{c \in p^{r-r'}\IZ/p^{r}\IZ} v^{H^1_{1,c}} +  \sum_{ \substack{ u \in (\IZ/p^{r'}\IZ)^\times \\ 0 \leq s \leq r-r' }}  v^{H^1_{1, up^s}} \\
                & = \sum_{a \in p \IZ/p^{r'} \IZ} v^{H^2_{a,1}}  +\sum_{c \in p^{r-r'}\IZ/p^{r}\IZ} v^{H^2_{1,c}} +  \sum_{ \substack{ u \in (\IZ/p^{r'}\IZ)^\times \\ 0 \leq s \leq r-r' }}   v^{H^2_{1, up^s}}
            \end{align*}            
            and this is up to multiplication by scalars the only non-trivial linear relation among the characteristic functions of types of one-dimensional cusps.
    \end{prop}

    As a corollary one can calculate the dimension of the space $\IC[L' / L]_{\Types(L)}$, see \thref{lem:space-types-dimension}. Moreover, since according to \cite[Proposition 5.1]{Bieker} the Borcherds products of invariant vectors are eta products, one obtains an identity of eta products which generalizes the identity found in \cite[Corollary 5.6]{Bieker}, see \thref{rem:EtaProductIdentity}.
    
    For a self-dual isotropic subgroup $H \subseteq L' / L$ let the special boundary divisor corresponding to $H$ be defined by
    $$Z(H) = \sum_{S} \#(H \cap \type_L(S)) [S],$$
    where the sum runs over all one-dimensional cusps of the compactified modular surface $X(L)$. We denote the space generated by these special boundary divisors by $\spbdiv(X(L))$. From the explicit description of the Borcherds lift of an invariant vector we obtain

    \begin{prop}[{see \thref{prop:boundary-isotropic-descent},  \thref{lem:boundary-mult-weyl-vect}}]
        The map
        $$\IQ[L' / L]^{\Mp_2(\IZ)} \to \spbdiv(X(L))_{\IQ}$$
        given by sending an invariant vector to the boundary divisor of its Borcherds product restricts to an isomorphism
        $$\IQ[L' / L]_{\Types(L)} \to \spbdiv(X(L))_{\IQ}.$$
        Under this map, the characteristic function of a self-dual isotropic subgroup $H$ is mapped to $Z(H)$.
    \end{prop}

    The proposition makes it possible to explicitly determine the boundary divisors of Borcherds products.

    \begin{thm}[{see \thref{thm:SpecialBoundaryDivisorsAsBorcherdsProducts}}]
        Let $f \in M_{0, L}^!$ and denote by $\Psi_f$ the corresponding Borcherds product. Then the boundary divisor of $\Psi_f(Z)$ on $X(L)$ is special and every special boundary divisor is the boundary divisor of $\Psi_\frakv(Z)$ for some $\frakv \in \IC[L' / L]_{\Types(L)}$.
        The space of special boundary divisors has dimension 
        $$\dim \spbdiv(X(L))_\IQ = \left(2((r-r'+1)p^{r'} - (r-r'-1) p^{r' - 1}) - 1\right) (r+1).$$
    \end{thm}

    In fact, we show in \thref{lem:boundary-mult-weyl-vect} that the boundary divisor of a Borcherds product only depends on a certain invariant vector whose corresponding Borcherds product has the same boundary divisor. This yields the previously mentioned converse theorem.

    We remark that there remain two obstacles for a full converse theorem. The first one is the weak converse theorem as in \cite[Theorem 4.23]{BruinierHabil}, which is only proven for signature $(2, n), n \geq 3$. We expect that such a theorem is still true also in our case. Then the proof of \cite[Theorem 4.2]{BruinierConverse} carries over to our situation so that the full converse theorem is equivalent to the injectivity of the Kudla--Millson lift. The second obstacle is that this injectivity is only known for $L = U \oplus U(N)$ so that even with the weak converse theorem, one only obtains the full converse theorem in these cases.
    
	\subsection*{Acknowledgements}

    We would like to thank Claudia Alfes, Jan Bruinier and Ma\-nu\-el Müller for helpful discussions.
	
	\section{Preliminaries}
	
	We briefly recall some facts on lattices, the Weil representation and associated modular forms. For more details also compare \cite{BruinierHabil, ScheithauerWeilRep}.
 
	\subsection{Vector valued modular forms for the Weil representation}

    Let $L$ be an even lattice, i.e. a finitely generated free $\IZ$-module with a non-degenerate quadratic form $q : L \to \IZ$. Then its discriminant group $L' / L$ is a so called \emph{discriminant form}, i.e. it is a finite abeleian group with a non-degenerate quadratic form $q \colon L' / L \to \IQ/\IZ$. We denote by $L^-$ the lattice with the same underlying $\IZ$-module and quadratic form given by $-q$. Then the discriminant form $(L' / L)^-$ of $L^-$ is given by the same underlying abelian group as $L' / L$ together with quadratic form $-q$. 
	Recall that the Weil representation $\rho_{L}$ is a representation of the metaplectic double cover $\Mp_2(\IZ)$ of $\SL_2(\IZ)$ on the group algebra $\IC[L' / L]$, see for example \cite[Section 1.1]{BruinierHabil}. We denote the standard basis of $\IC[L' / L]$ by $\frake_\gamma$ for $\gamma \in L' / L$. 
    We denote by $\IC[L' / L]^{\Mp_2(\IZ)}$ the space of invariants under $\rho_L$.
	
	Let $k \in \frac 1 2 \IZ$. A holomorphic function $f \colon \IH \to \IC[L' / L]$ is a \emph{weakly holomorphic modular form of weight $k$ for $\rho_L$} if $f(M\tau) = \phi^{2k}(\tau) \rho_L(M, \phi) f(\tau)$
	for all $(M, \phi) \in \Mp_2(\IZ)$ and $f$ is meromorphic at $\infty$. If $f$ is holomorphic at $\infty$ (respectively vanishes at $\infty$), $f$ is called a \emph{holomorphic modular form} (respectively \emph{cusp form}) of weight $k$ for $\rho_L$.
	We denote by $M_{k,L}^!$ (respectively $M_{k,L}$ and $S_{k,L}$) the space of weakly holomorphic modular forms (respectively holomorpic modular forms and cusp forms) of weight $k$ for $\rho_L$. 

    Let
    $$\Delta_k = -v^2 \left(\frac{\partial^2}{\partial u^2} + \frac{\partial^2}{\partial v^2}\right) + ikv \left(\frac{\partial}{\partial} + i\frac{\partial}{\partial v}\right), \quad \tau = u+iv \in \IH$$
	be the weight $k$ Laplace operator on $\IH$. Following \cite{BruinierFunke} a \emph{harmonic weak Maass form} of weight $k$ for $\rho_L$ is a real analytic function $f \colon \IH \to \IC[L' / L]$ that has the transformation property of modular forms, satisfies $\Delta_k f = 0$ and has at most linear exponential growth at $\infty$. 
	Let $\xi_k$ be the linear differential operator on harmonic Maass forms given by $\xi_k(f) = 2i v \overline{\frac{\partial}{\partial \bar \tau} f}$. Then, by the linear exponential growth, $\xi_k(f)$ is a weakly holomorphic modular form of weight $2-k$ for the dual Weil representation $\rho_L^* = \rho_{L^-}$ of $\rho_L$. 
	We denote by $H_{k,L}$ the space of harmonic weak Maass forms of weight $k$ for $\rho_L$ such that $\xi_k(f)$ is a cusp form.
	Note that the map $\xi_k \colon H_{k,L} \to S_{2-k, L^-}$ is surjective with kernel $M_{k,L}^!$, see \cite[Corollary 3.8]{BruinierFunke}. Every harmonic weak Maass form $f \in H_{k,L}$ has a Fourier expansion of the form
    \begin{align*}
        f(\tau)
        &= \sum_{\gamma \in L' / L} \sum_{\substack{n \in \IQ \\ n \gg -\infty}} c^+(\gamma, n) e(n \tau) + \sum_{\gamma \in L' / L} \sum_{\substack{n \in \IQ \\ n < 0}} c^-(\gamma, n) \Gamma(1 - k, -4 \pi n v) e(n \tau),
    \end{align*}
    where $e(x) = e^{2 \pi i x}$. The finite sum
    $$P(f) = \sum_{\gamma \in L' / L} \sum_{n \leq 0} c^+(\gamma, n) e(n \tau)$$
    is called the principal part. Examples of harmonic weak Maass forms are given by the Hejhal Poincar\'e series $F_{\beta, m}(\tau, 1 - k/2)$, see \cite[Section 1.3]{BruinierHabil}. For $k = 0$ these have to be defined using Hecke summation and we have the following

    \begin{lem}\label{lem:HarmonicMassFormDecomposition}
        Let $f \in H_{0, L}$ with principal part
        $$P(f) = \sum_{\gamma \in L' / L} \sum_{n \leq 0} c^+(\gamma, n) e(n \tau).$$
        Then there exists an invariant vector $\frakv_f \in \IC[L' / L]^{\SL_2(\IZ)}$ such that
        $$f(\tau) = \frakv_f + \sum_{\gamma \in L' / L} \sum_{n < 0} c^+(\gamma, n) F_{\beta, m}(\tau, 1).$$
    \end{lem}

    \begin{proof}
        The difference
        $$f(\tau) - \sum_{\gamma \in L' / L} \sum_{n < 0} c^+(\gamma, n) F_{\beta, m}(\tau, 1)$$
        is a harmonic weak Maass form with constant principal part, hence is a holomorphic modular form of weight $0$ according to \cite[Lemma 3.3]{BruinierYang}, i.e. an invariant vector $\frakv_f \in \IC[L' / L]^{\SL_2(\IZ)}$.
    \end{proof}

    For an isotropic subgroup $H$ of $L' / L$ the lattice $L_H = L + H$ is also an even lattice with dual lattice $L_H' = L + H^\perp$. The corresponding discriminant form is $L_H' / L_H = H^\perp / H$. Conversely, when $\tilde L \subset L$ is a sublattice of finite index the quotient $H = L/\tilde L$ is an isotropic subgroup of $\tilde L' / \tilde L$ with orthogonal complement $H^\bot = L'/\tilde L$ and $H^\bot/H \cong L'/L$.
    
    Recall that for an isotropic subgroup $H \subseteq L' / L$ we have a pair of operators, compare \cite[Section 5.1]{BruinierHabil} and \cite[Section 4]{ScheithauerConstructions}, the \emph{isotropic induction}
    $$ \uparrow_H^L \colon  \IC[H^{\perp}/H] \to \IC[L' / L], \qquad \uparrow_H^L(\frake_{\gamma+H}) = \sum_{\gamma' \in H} \frake_{\gamma + \gamma'}   $$
    for $\gamma\in H^{\perp}$ and \emph{isotropic descent}
    $$ \downarrow_H^L \colon \IC[L' / L] \to \IC[H^{\perp}/H], \qquad \downarrow_H^L(\frake_{\gamma}) = 
    \begin{cases}
     \frake_{\gamma+H} & \text{if $\gamma\in H^{\perp}$}, \\
              0 & \text{otherwise}.
    \end{cases} 
    $$
    The two operators commute with the Weil representation and therefore extend to maps $\uparrow_H^L \colon H_{k,L_H} \to H_{k,L}$ and $\downarrow^L_H \colon H_{k,L} \to H_{k, L_H}$. In particular, the two operators map invariant vectors to invariant vectors.
    Moreover, they are adjoint for the Petterson inner product (which restricts to the standard inner product on invariants up to a constant). If $\tilde{L} \subseteq L$ is a sublattice of finite index with corresponding isotropic subgroup $H$, we also write $\uparrow_L^{\tilde L}$ and $\downarrow_L^{\tilde L}$.
    
	\subsection{The lattice \texorpdfstring{$L = U(N) \oplus U(N')$}{L = U(N) + U(N')}}
	   
	For a lattice $(L,q)$ and positive integer $N$ we denote by $L(N)$ the lattice with rescaled quadratic form $Nq$. Let $U = \IZ^2$ with $q(x, y) = xy$ be the hyperbolic plane.
    The main example we consider below is the orthogonal sum of two rescaled hyperbolic planes $L = L_{N,N'} = U(N) \oplus U(N')$ for two positive integers $N, N'$ such that $N' \mid N$.
	A model for this lattice is given by
	\begin{align*}
		L = \left\{ \begin{pmatrix} a & b \\ c & d \end{pmatrix} \in \M_2(\IZ) \bigg\vert \frac{N}{N'} \mid c \right\}
	\end{align*}
	with quadratic form $q(X) = -N' \det(X)$. Let $V = L \otimes \IQ$ be its associated rational quadratic space and $V_\IR = L \otimes \IR$ the associated real quadratic space given by real $2\times2$ matrices with quadratic form $-N'\det$. The hermitian symmetric domain associated to the orthogonal group $O(V)$ can be realised inside $\IP(V \otimes \IC)$ as the open subset in the zero quadric $\calN$ given by $\calK = \{ [Z] \in \calN \colon (Z, \bar Z) > 0\}$. More explicitly, the map
    \begin{align}
		\label{eq:identification-upper-half-plane}
		 \IH \times \IH & \to \calK \\
		 \nonumber
		 Z = (z_1, z_2) & \mapsto \begin{bmatrix} z_1 & -z_1 z_2 \\ \frac{1}{N'} & -z_2 \end{bmatrix} =: Z_L
	 \end{align}
  biholomorphically identifies $\IH \times \IH$ with one of the two connected components of $\calK$ that we denote by $\calH$. Note that the action of $O(V_\IR)$ on $\calK$ restricts to an action of the identity component $\SO^+(V_\IR) \subseteq O(V_\IR)$ on the connected component. Moreover, in our setting the double cover $\Spin(V_\IR) \to \SO^+(V_\IR)$ can be identified with $\SL_2(\IR) \times \SL_2(\IR)$ such that the induced action on $\IH \times \IH$ is given by componentwise Möbius transformations.
 
 The dual lattice of $L$ is given by
	\begin{align*}
		L' = \left\{ \begin{pmatrix} a' & b' \\ c' & d' \end{pmatrix} \in \M_2(\IQ) \bigg\vert N'a', Nb', N'c', N'd' \in \IZ \right\}.
	\end{align*}
    The discriminant group $L'/L$ is identified with $(\IZ/N\IZ)^2 \oplus (\IZ/N'\IZ)^2$ using the map
    $$\begin{pmatrix} a' & b' \\ c' & d' \end{pmatrix} \mapsto \left(Nb', \frac{N}{N'}c', N'a', -N'd'\right),$$
    where the quadratic form on $(\IZ/N\IZ)^2 \oplus (\IZ/N'\IZ)^2$ is given by $q(w,x,y,z) = \frac{wx}{N} + \frac{yz}{N'}$.
	
	Under the double cover $\psi : \SL_2(\IR) \times \SL_2(\IR) \to \SO^+(V_\IR)$ the discriminant kernel $\Gamma(L)$ is identified as 
		\begin{align*}
			\Gamma(L) \simeq &\bigg\{ \left(\begin{pmatrix} a_1 & N' b_1 \\ N c_1 & d_1 \end{pmatrix}, \begin{pmatrix} a_2 & N' b_2 \\ N c_2 & d_2 \end{pmatrix} \right) \bigg\vert \ a_i, b_i, c_i, d_i \in \IZ \text{ for } i=1,2, \\
			& \quad a_1 a_2 \equiv d_1 d_2 \equiv 1 \! \pmod{N}, \ a_1 d_2 \equiv a_2 d_1 \equiv 1 \!\pmod{N'} \bigg\} / \{\pm 1\},
		\end{align*}
        compare \cite[Proposition 2.4]{Bieker} (note that we chose a slightly different model of $L_{N,N'}$ here).

        Let $\Gamma \subseteq O(V)$ be a subgroup that is commensurable with $\Gamma(L)$. The Baily-Borel compactification of $\Gamma \backslash \IH \times \IH$ is given as follows.
	For an isotropic vector $v \in V$ the point $[v] \in \IP(V_\IC)$ is in the boundary of $\mathcal{H}$ and we call $\{[v]\}$ a \emph{zero-dimensional (rational) boundary component}, \emph{special boundary point} or \emph{cusp}. For a two-dimensional isotropic subspace $F$ of $V$, we call the set of all non-special boundary points which can be represented by elements of $F \otimes_\IQ \IC$ the \emph{one-dimensional (rational) boundary component} or the \emph{one-dimensional (rational) cusp associated to $F$}. 
	The cusps of $\IH \times \IH$ are in bijection with $\IP^1(\IQ) \times \IP^1(\IQ)$ via $\begin{bsmallmatrix}
		t & -st \\
		1 & -s
	\end{bsmallmatrix} \mapsto (t,s)$, e.g. the isotropic vector $e_1$ represents the cusp $(\infty, \infty)$.
	The one-dimensional rational boundary components are via \eqref{eq:identification-upper-half-plane} of the form $\{ a/c\} \times \IH$ or $\IH \times \{ a/c\}$ for $a/c \in \IP^1(\IQ)$, where $\{ a/c\} \times \IH$ corresponds to the isotropic plane $I^1_{a,c} = \langle \begin{psmallmatrix} a & 0 \\ c & 0 \end{psmallmatrix}, \begin{psmallmatrix} 0 & a \\ 0 & c \end{psmallmatrix} \rangle \subset V$ while $\IH \times \{ a/c\}$ corresponds to the isotropic plane $I^2_{a,c} = \langle \begin{psmallmatrix} 0 & 0 \\ c & a \end{psmallmatrix}, \begin{psmallmatrix} c & a \\ 0 & 0 \end{psmallmatrix} \rangle \subset V$ for $a/c \in \IP^1(\IQ)$ with $a, c \in \IZ$ such that $\gcd(a,c) = 1$ and $c \geq 0$.
	We denote the union of $\IH \times \IH$ with all rational boundary components by $(\IH \times \IH)^*$ and write $X_{\Gamma} = \Gamma \backslash (\IH \times \IH)^*$.
	Then $X_{\Gamma}$ admits a structure of a normal projective variety and $Y_{\Gamma} = \Gamma \backslash ( \IH \times \IH) \subset X_{\Gamma}$ is an open subvariety.
	Note that the boundary $X_{\Gamma} \setminus Y_{\Gamma}$ has codimension one. For $\Gamma = \Gamma(L)$ we will use the notation $X(L) = X_{\Gamma(L)}$ and $Y(L) = Y_{\Gamma(L)}$.

        A \emph{(meromorphic) modular form} for $\Gamma$ of weight $r \in \IQ$ and multiplier system $\chi$ is a (meromorphic) function $f : \IH \times \IH \to \IC$ with the transformation property
    $$f(M_1 z_1, M_2, z_2) = (c_1 z_1 + d_1)^r (c_2 z_2 + d_2)^r \chi(M_1, M_2) f(z_1, z_2)$$
    for $M_1, M_2 \in \Gamma$, where we identify $\Gamma$ with its preimage in $\SL_2(\IR) \times \SL_2(\IR)$ and $M_i = \begin{psmallmatrix}
		a_i & b_i \\ c_i & d_i
	\end{psmallmatrix}$, and such that $f$ is meromorphic along the boundary.
	Note that since the Witt rank of $V$ is $2$ (or equivalently since the boundary has codimension one), the Koecher principle is not available. We call a meromorphic modular form $f : \IH \times \IH \to \IC$ a \emph{modular unit}, if its zeros and poles are concentrated on the boundary.
        
    For further reference we record the following two lemmas.
    \begin{lem}
        \thlabel{lem:sublattice-LNN}
        Let $L$ be an even lattice of signature $(2,2)$, Witt rank 2 and level $N$ as before. Then $L$ contains a finite index sublattice of the form $L_{N,N'}$ for some divisor $N'$ of $N$.
        
        Assume moreover that $L$ splits a scaled hyperbolic plane $U(N'')$ for some divisor $N''$ of $N$. 
        Then $L$ contains a sublattice of finite index isomorphic to $L_{N,N''}$.
    \end{lem}

    As an immediate consequence of the previous lemma together with the fact that for lattices $L = L_{N,N'}$ the discriminant form $L'/L$ admits self-dual isotropic subgroups, we obtain the following
    \begin{lem}[{\cite{Bieker}}]
        Let $L$ be an even lattice of signature $(2,2)$ and of Witt rank 2. Then $L'/L$ admits self-dual isotropic subgroups and $\IC[L'/L]^{\SL_2(\IZ)}$ is generated by characteristic functions of self-dual isotropic subgroups.
    \end{lem}

    \section{The type of one-dimensional cusps of $X(L_{N, N'})$}\label{sec:OneDimCusps}

    We study one-dimensional rational boundary components in the case $L = L_{N,N'}$ explicitly.
    Note that in the general setting one-dimensional cusps were classified e.g. by \cite{AttwellDuval, KieferDiss, DriscollSpittlerScheithauerWilhelm}.
    We use the following notation of \cite{DriscollSpittlerScheithauerWilhelm}.
    Let $L$ be any even lattice of signature $(2,2)$ and Witt rank 2. For an isotropic plane $I \subseteq V$ the \emph{type} of $I$ is the isotropic subgroup
    $$ \mathrm{type}_L(I) = (I \cap L')/(I \cap L)$$
    of the discriminant group $L'/L$.
    As the discriminant kernel acts trivially on $L'/L$ rational isotropic planes representing the same one-dimensional cusp of $X(L)$ have the same type.
    We define the type of a one-dimensional cusp to be the type of its representatives.
    Moreover, we denote by $\mathrm{Types}(L)$ the set of all types of isotropic planes in $V$ (or equivalently the set of types of all one-dimensional cusps of $X(L)$).

    We will use the following setup, compare \cite[Section 2.1]{BruinierHabil}.
    Let $z$ and $\tilde z$ be an (ordered) pair of primitive isotropic vectors of $L$ that span $I$ over $\IQ$.
    Let us moreover fix $z' \in L'$ with $(z,z') = 1$ and $(\tilde z, z') = 0$.
    We set $K = L \cap z^\bot \cap (z')^\bot$. 
    Let $N_z$ be the positive integer such that $(z,L) = N_z \IZ$ and similarly let $N_{\tilde z}$ be the positive integer such that $(\tilde z, K) = N_{\tilde z} \IZ$, called the \emph{level} of $z$ and $\tilde z$, respectively.
    In particular, $\frac{1}{N_z} z \in L'$ and $\frac{1}{N_{\tilde z}} \tilde z \in K'$.
    Moreover, we choose $\tilde z' \in K'$ such that $(\tilde z, \tilde z') = 1$.

    Let $L_0$ be the lattice generated by $L$ and $\frac{z}{N_z}$ and let $L_{00}$ be the lattice generated by $L_0$ and $\frac{\tilde z}{N_{\tilde z}}$. Similarly, let $K_0$ be the lattice generated by $K$ and $\frac{\tilde z}{N_{\tilde z}}$. Note that
	$$L_0' = \{ \lambda \in L' \ \vert\ (\lambda, z) = 0 \pmod{N_z} \}.$$

    Modifying the orthogonal projection from $L'$ to $K'$ as in \cite[(2.7)]{BruinierHabil} we obtain a projection $p_L \colon L_0' \to K'$.
    
    \begin{lem}
        \thlabel{lem:projection-p-isos}
        The map $p_L$ induces an isomorphism
        $$ L_0'/L_0 \xrightarrow[p_L]{\cong} K'/K.$$
        Moreover, $|L'/L| = N_z^2|K'/K| = N_z^2N_{\tilde z}^2$.
    \end{lem}
    \begin{proof}
        This is contained in \cite[Section 2.1, p.41]{BruinierHabil}.
    \end{proof}
    \begin{cor}
        Let $L$ be an even lattice of signature $(2,2)$ and Witt rank 2 and let $I \subset V$ be an isotropic plane.    
        The type of $I$ is a self-dual isotropic subgroup of $L'/L$, i.e. $\type_L(I)^\perp = \type_L(I)$ 
    \end{cor}
    \begin{proof}
        The type of $I$ is generated by the images of $\frac{z}{N_{z}}$ and $\frac{\tilde z}{N_{\tilde z}}$, so $|\mathrm{type}(I)| = N_zN_{\tilde z}$.
        But in this case $M = \{0\}$, so $|L'/L| = N_z^2N_{\tilde z}^2$.
        Hence, $\mathrm{type}_L(I)$ is a self-dual isotropic subgroup of $L'/L$.
    \end{proof}
    
    We denote by $\IC[L'/L]_{\mathrm{Types}(L)} \subseteq \IC[L'/L]^{\SL_2(\IZ)}$ the subspace generated by the characteristic functions $v^H$ where $H$ ranges over $\mathrm{Types}(L)$.
    Recall that $L' / L$ decomposes as an orthogonal direct sum $L' / L \cong \bigoplus_{p \mid N} D_p$ of its $p$-components, where $N$ is the level of $D$.
    Moreover, the self-dual isotropic subgroups $H$ of $L' / L$ are precisely those of the form $H \cong \bigoplus_{p \mid N} H_p$ where $H_p$ is a self-dual isotropic subgroup of $D_p$ by \cite[Lemma 3.3]{Bieker}.
    For $p \mid N$ we denote by $\mathrm{Types}(L)_p$ the set of $H_p$ for $H \in \mathrm{Types}(L)$.
    
    Let $L \subset L^{(p)} \subset L'$ be such that $(L^{(p)})'/L^{(p)} = D_p$, e.g. $L^{(p)} = L + H$ for the isotropic subgroup $H = \bigoplus_{p' \neq p} H_{p'}$ of $L' / L$ where $H_{p'} \subset D_{p'}$ is self-dual isotropic for every prime factor $p' \neq p$ of $N$.
    \begin{lem}
    \thlabel{lem:types-pcomponent}
        Let $I$ be an isotropic plane in $V$. Then $\mathrm{type}_L(I)_p = \mathrm{type}_{L^{(p)}}(I)$ as subgroups of $\left(L'/L\right)_p = \left(L^{(p)} \right)'/L^{(p)}$.
        In particular, $\mathrm{Types}(L)_p = \mathrm{Types}(L^{(p)})$.
    \end{lem}
    \begin{proof}
        By construction the inclusion $I \cap \left(L^{(p)}\right)' \subset I \cap L'$ induces an injection $\mathrm{type}_{L^{(p)}}(I)\hookrightarrow \mathrm{type}_L(I)_p$. But since both are self-dual isotropic subgroups of $(L'/L)_p$ they have to agree.
    \end{proof}

    In the case $L = L_{N,N'}$ we can classify the types for $L$ more explicitly.
    Recall that the rational boundary component $\{\frac{a}{c}\} \times \IH$ (where $a,c \in \IZ$ with $\gcd(a,c) = 1$, $c \geq 0$ as above) corresponds to the isotropic plane $I^1_{a,c} \subset V$ for which we choose the basis
        $I^1_{a,c} = \left \langle z, \tilde z \right \rangle$  with 
        $$z = \begin{pmatrix} 0 & a \\ 0 & c \end{pmatrix} \qquad \text{ and } \qquad \tilde{z} = \frac{N}{N'M} \begin{pmatrix} a & 0 \\ c & 0  \end{pmatrix},$$
    where $M = \gcd(N/N',c)$. Note that both $z$ and $\tilde z$ are primitive isotropic vectors of $L$.
        Similarly, for the isotropic plane $I^2_{a,c}$ corresponding to the rational boundary component $\IH \times \{\frac a c \}$ we choose the basis $I^2_{a,c} = \left \langle z, \tilde z \right \rangle$ with
        $$z = \begin{pmatrix} c & a \\ 0 & 0 \end{pmatrix} \qquad \text{ and } \qquad \tilde{z} = \frac{N}{N'M} \begin{pmatrix} 0 & 0 \\ c & a  \end{pmatrix}.$$
    
    We determine the types of the isotropic planes as follows. For $* = 1,2$ we set $H^*_{a,c} = \mathrm{type}(I^*_{a,c})$.
    \begin{lem}
        \thlabel{lem:isotropic-subgroups-cusps}
        Let $L = L_{N,N'}$.
        The intersection $I^*_{a,c} \cap L'$  is generated by $\frac{z}{N_z}$ and $\frac{\tilde z}{N_{\tilde z}}$.
        Its type $H^*_{a,c}$ is the self-dual isotropic subgroup of $L'/L \cong (\IZ/N\IZ)^2 \oplus (\IZ/N'\IZ)^2$ generated by 
        $$ \left(\frac{Na}{N'M},0,0,-\frac c M \right) \qquad \text{and} \qquad \left(0,c,a,0\right) $$
        for $* = 1$ and for $* = 2$ by 
        $$ \left(\frac{Na}{N'M},0,\frac c M,0 \right) \qquad \text{and} \qquad \left(0,c,0,-a\right) .$$
        In the notation of \cite[Section 4]{Bieker} we have 
        $H^1_{a,c} = H_{(\frac{N}{N'M}d_1, 0, Md_2), (d_1, 0, d_2)}^{(0,-u d_2),(u^{-1} d_1, 0)}$
        where $d_1 = \gcd(a,N')$, $d_2 = \gcd(c/M,N')$ and $u \in (\IZ/N'\IZ)^\times$ such that $\left(\frac{Na}{N'M},0,0,-\frac c M \right)$ is a multiple of $\left( \frac{N}{N'M} d_1, 0, 0, -ud_2\right)$.
        Similarly, $H^2_{a,c} = H_{(\frac{N}{N'M}d_1, 0, Md_2), (d_2, 0, d_1)}^{(-u d_2, 0),(0, u^{-1} d_1)}$.
    \end{lem}
    \begin{proof}
        It can be verified quickly that the vectors in the statement of the lemma are precisely the images of $\frac{z}{N_z}$ and $\frac{\tilde z}{N_{\tilde z}}$ under the identification $L'/L \cong (\IZ/N\IZ)^2 \oplus (\IZ/N'\IZ)^2$.
        The last statement follows from the construction, namely the subgroup $H_{(\frac{N}{N'M}d_1, 0, Md_2), (d_1, 0, d_2)}^{(0,-u d_2),(u^{-1} d_1, 0)}$ is by definition generated by $\left( \frac{N}{N'M} d_1, 0, 0, -ud_2\right)$ and $(0, Md_2, u^{-1}d_1,0)$.
    \end{proof} 
    
    \begin{rem}
        Certain self-dual isotropic subgroups of $L'/L$ are constructed and classified in \cite[Section 4]{Bieker}.
        In particular, when $N = N' = p$ for some prime number $p$ every self-dual isotropic subgroup is of this form by \cite[Proposition 4.10]{Bieker}.
        In general this will not be true: already for $N = N' = p^2$ not every self-dual isotropic subgroup from \cite[Example  4.7]{Bieker} is of the form $H^*_{a,c}$ as above (and there are also self-dual isotropic subgroups not contained in the list of \cite[Example 4.7]{Bieker} in this case).
    \end{rem}

    \begin{lem}
    \thlabel{lem:types-factorisation}
        Let $L$ be a lattice of signature $(2,2)$, Witt rank 2 and level $N$. Then the natural maps 
        $$\mathrm{Types}(L) \xrightarrow{\cong} \prod_{p \mid N} \mathrm{Types}\left(L^{(p)} \right) $$
        and
        $$ \IC[L'/L]_{\mathrm{Types}(L)} \xrightarrow{\cong} \bigotimes_{p \mid N}\IC\left[\left( \left(L^{(p)}\right)'/L^{(p)}\right)\right]_{\mathrm{Types}\left( L^{(p)} \right)}$$
        are isomorphisms.
    \end{lem}
    \begin{proof}
        We first show the bijectivity of the first map.
        The map that associates to a self-dual isotropic subgroup $H$ its $p$-primary parts $(H_p)_{p \mid N}$ is clearly injective so it remains to show surjectivity of the first map.
        By \thref{lem:sublattice-LNN} we may choose a sublattice $L_{N,N} \subseteq L$ of finite index.
        Note that for two rational isotropic planes $I_1$ and $I_2$ such that $\mathrm{type}_{L_{N,N}}(I_1) = \mathrm{type}_{L_{N,N}}(I_2)$ also $\mathrm{type}_{L}(I_1) = \mathrm{type}_{L}(I_2)$. Hence,
        \begin{align*}
            \mathrm{Types}(L_{N,N}) & \to \mathrm{Types}(L) \\
            \mathrm{type}_{L_{N,N}}(I) &\mapsto \mathrm{type}_{L}(I)
        \end{align*}
        is a well-defined map which is surjective by construction.
        It thus suffices to show surjectivity for $L = L_{N,N}$.
        By \thref{lem:isotropic-subgroups-cusps}, for fixed $* \in \{ 1,2 \}$ the types of $L_{N,N}$ of the form $H^*_{a,c}$ are in bijection with cyclic subgroups of $(\IZ/N\IZ)^2$ of order $N$. The claim now follows from the Chinese Remainder Theorem.
        
        By \cite[Proposition 3.1]{EhlenSkoruppa} the decomposition of $L'/L$ into its $p$-primary parts induces an isomorphism 
        $$\IC[L'/L]^{\SL_2(\IZ)} \cong \bigotimes_{p \mid N} \IC[(L'/L)_p]^{\SL_2(\IZ)} \cong \bigotimes_{p \mid N} \IC\left[\left( \left(L^{(p)}\right)'/L^{(p)}\right)\right]^{\SL_2(\IZ)} $$ 
        which restricts to an isomorphism on the subspaces generated by types by the first assertion.      
    \end{proof}

    \begin{prop}
    \thlabel{lem:space-types}
        Let $L = L_{N,N'}$ for $N' \mid N$ with prime factorizations $N = \prod_{p \mid N} p^{r_p}$ and $N' = \prod_{p \mid N} p^{r_p'}$.
            Then 
            $$|\mathrm{Types}(L)| = \prod_{p \mid N'} 2((r_p-r_p'+1)p^{r_p'} - (r_p-r_p'-1) p^{r_p' - 1}) \cdot \prod_{p \nmid N'}  ({r_p}+1). $$
            Moreover, when $N = p^r$ and is a prime power and $N' \neq 1$ the different types of $L$ are precisely
            $H^*_{pa, 1}$ for $a \in \IZ/p^{r'-1} \IZ$, $H^*_{1, c p^{r-r'+1}}$ for ${c} \in \IZ/p^{r'-1} \IZ$ and $H^*_{1, up^s}$ for $1 \leq u \leq p^{r'}$ with $\gcd(u, p) = 1$ and $0 \leq s \leq r-r'$.
    \end{prop}
    
    \begin{proof}
        By \thref{lem:types-factorisation} it suffices to consider the case $N = p^r$ and $N' = p^{r'}$.
        As $\gcd(a,c) = 1$, we have $\gcd(a,p) = 1$ or $\gcd(c,p) = 1$.
        As $H^*_{a,c}$ does not change when we multiply both $a$ and $c$ with a unit in $\IZ/p^r \IZ$, every type is of the form $H^*_{1, up^s}$ for some $0 \leq s \leq r$ and $u \in (\IZ/p^{r'} \IZ)^\times$ or $H^*_{up^s, 1}$ for some $0 \leq s \leq r'$ and $u \in (\IZ/p^{r'} \IZ)^\times$.
        
        In order to count the number of types we first note that when $N' = 1$ we always have $H^1_{a,c} = H^2_{a,c}$. We claim that when $N' \neq 1$ a subgroup of the form $H^1_{a,c}$ is never of the form $H^2_{a,c}$ and vice versa.
        In order to show this claim we distinguish three cases.
        Let first $\gcd(c,p) = 1$. Then $H^1_{a,c} = H^1_{up^s,1}$ by the discussion above.
        Now if $(0, 1, u p^s, 0) \in H^2_{a',c'}$ for some $a',c'$, then $\gcd(c',p) = 1$ and $a = a' = 0$. But then $(0,0,0,1) \not \in H^2_{a',c'}$ if $N' > 1$. 
        Let now $\gcd(a,p) = 1$, i.e. $H^1_{a,c} = H^1_{1, up^s}$ and assume $s \leq r -r'$. In this case if $(0, u p^s, 1, 0), (u^{-1}p^{r-r'-s}, 0 ,0 ,-1) \in H^2_{a',c'}$ for some $a',c'$, we find $\gcd(a', p) = 1$ and $c' = 0$, hence $c = 0$ which gives a contradiction.
        Lastly, assume $\gcd(a,p) = 1$ and $s > r -r'$. If in this case $H^1_{1, up^s} = H^2_{a',c'}$ for some $a',c'$, we find $\gcd(a',p) = p^{s+r'-r} > 1$ so $\gcd(c', p) = 1$, which implies $p^s = \gcd(c, p^r) = 1$ contradicting our assumption.

        The remaining claims in the $N' = 1$ case are \cite[Corollary 4.3]{Bieker}. Let us thus assume $r' \geq 1$.
        By the claim in the previous paragraph it suffices to count the subgroups of the form $H^1_{a,c}$ in order to establish the assertion.
        Let us first assume that $p^s = \gcd(c,p^r) \leq p^{r-r'}$ and $a = 1$. In this case $H^1_{1,up^s}$ is generated by $(p^{r-r'-s},0,0,-u)$ and $(0,p^s,u^{-1},0)$. In particular $H^1_{1,up^s} = H^1_{1, u'p^{s'}}$ for some $u'$ and $s'$ if and only if $u = u'$ in $(\IZ/p^{r'}\IZ)^\times$ and $s = s'$. 
        When $c = p^s$ for $s > r-r'$ (which implies $a = 1$), we find $H^1_{1,up^s} = H^1_{1, u'p^{s'}}$ if and only if $u = u'$ in $(\IZ/p^{r-s}\IZ)^\times$ and $s = s'$ as $H^1_{1,up^s}$ is generated by $(1,0,0, - up^{r'-(r-s)})$ and $(0, up^s, 1,0)$.
        Similarly, $H^1_{up^s,1} = H^1_{u'p^{s'},1}$ if and only if $s = s'$ and $u = u'$ in $(\IZ/p^{r' - s} \IZ)^\times$. Hence there are precisely $p^{r'}$ subgroups of this form.
        Moreover, $H^1_{1, up^s} = H^1_{u'p^{s'}, 1}$ if and only if $s = s' = 0$ and $u = (u')^{-1}$.
    \end{proof}

    \begin{prop}
            \thlabel{lem:space-types-relation}
            If $N = p^r$ is a prime power, then
            \begin{align}
                 & \sum_{a \in p \IZ/p^{r'} \IZ} v^{H^1_{a,1}}  +\sum_{c \in p^{r-r'}\IZ/p^{r}\IZ} v^{H^1_{1,c}} +  \sum_{ \substack{ u \in (\IZ/p^{r'}\IZ)^\times \\ 0 \leq s \leq r-r' }}  v^{H^1_{1, up^s}} \nonumber \\
                & = \sum_{a \in p \IZ/p^{r'} \IZ} v^{H^2_{a,1}}  +\sum_{c \in p^{r-r'}\IZ/p^{r}\IZ} v^{H^2_{1,c}} +  \sum_{ \substack{ u \in (\IZ/p^{r'}\IZ)^\times \\ 0 \leq s \leq r-r' }}   v^{H^2_{1, up^s}}
                \label{lem:space-types-relation-eq}
            \end{align}            
            and this is up to multiplication by scalars the only non-trivial linear relation among the characteristic functions of types of one-dimensional cusps.      
    \end{prop}
    Note that when $N' = p$ this is essentially contained in \cite[Proposition 4.12]{Bieker}. The linear relation \eqref{lem:space-types-relation-eq} is the sum of the $r$-many linear relations of \cite[Proposition 4.12]{Bieker}.  
    \begin{proof}
        The assertion follows as in the proof of \cite[Proposition 4.12]{Bieker}. Let us sketch the argument. 
        We first verify that the linear relation \eqref{lem:space-types-relation-eq} holds by checking that every element of $L'/L$ is contained in as many subgroups appearing on the left hand side of the equation as on the right hand side. For the zero vector this is clear. 

        We first consider elements of $H^1_{a,1}$ for $a \in p \IZ/p^{r'} \IZ$. Let us now consider $\lambda = x(ap^{r-r'}, 0, 0, -1) + y(0,1,a,0)$ with $x \in \IZ/p^{r'} \IZ$ and $y \in \IZ/p^r \IZ$ with $\lambda \neq 0$. Let $r_x$ such that $\gcd(p^{r'}, x) = p^{r_x}$ and similarly define $r_y$. Then $\lambda$ is contained precisely in $H^1_{\tilde a,1}$ with $xa = x\tilde a$ and $ya = y \tilde a$. Thus there are precisely $p^{\min(r_x, r_y)}$ such subgroups. Moreover, $\lambda \not \in H^1_{1,c}$ for any $c$.
        When $r_y < r_x$ we find $\lambda \in H^2_{\tilde a, 1}$ for all $\tilde a \in \IZ/p^{r'} \IZ$ with $y\tilde a = x$ in $\IZ/p^{r'} \IZ$ and when $r_x \leq r_y$ we have $\lambda \in H^2_{1,\tilde c}$ with $x \tilde c = y$.
        The other cases can be checked similarly.

        In order to see that \eqref{lem:space-types-relation-eq} is the only relation (up to scalar multiplication) we note that for every pair of subgroups $H^1_{*,*}$ and $H^2_{*,*}$ appearing on the left hand side respectively on the right hand side of the equation we found an element of $L'/L$ appearing in precisely these two subgroups.
    \end{proof}

    \begin{cor}
        \thlabel{lem:space-types-dimension}
        Let $L = L_{N,N'}$ as above. Then the dimension of $ \IC[L'/L]_{\mathrm{Types}(L)}$ is given by 
            \begin{equation}
                \prod_{p \mid N'} \left(2((r_p-r_p'+1)p^{r_p'} - (r_p-r_p'-1) p^{r_p' - 1}) - 1\right)  \cdot \prod_{p \nmid N'}  ({r_p}+1).
                \label{lem:space-types-dimension-eq}
            \end{equation}  
            The characteristic functions of types span the full space of invariants if and only if $N$ is square-free.        
    \end{cor}
    \begin{proof}
        The dimension formula follows directly from \thref{lem:space-types} together with \thref{lem:space-types-relation} and \cite[Corollary 4.3]{Bieker}. The last assertion follows by comparing \eqref{lem:space-types-dimension-eq} with the dimension formula for $\IC[L'/L]^{\SL_2(\IZ)}$ of \cite[Theorem 5.4, Discussion after Corollary 5.5]{Zemel2021}.
    \end{proof}
	
	\section{Borcherds products in signature \texorpdfstring{$(2,2)$}{(2,2)}}

	Let us recall the construction of Borcherds products following \cite{Borcherds, BruinierHabil}. Let again $L$ be an even lattice of signature $(2,2)$ and Witt rank 2. We consider the theta function
	\begin{align*}
		\Theta_L(\tau, Z) &= v \sum_{\lambda \in L'} e\left(\frac{iv}{2 y_1 y_2} \lvert (\lambda, Z_L) \rvert^2\right) \frake_\lambda(q(\lambda) \overline{\tau})
	\end{align*}

    and define the regularized theta lift of a harmonic weak Maass form $f \in H_{0, L}$ by
    $$\Phi(Z, f) = \int_{\Mp_2(\IZ) \bs \IH}^{\reg} \langle f(\tau), \Theta_L(\tau, Z) \rangle \frac{du dv}{v^2}.$$

    Recall the setup from the beginning Section \ref{sec:OneDimCusps}.
    We fix a primitive isotropic vector $z \in L$ together with $z' \in L'$ such that $(z,z') = 1$. 
	We set $K = L \cap z^\bot \cap (z')^\bot$ and get a projection $p_L \colon L_0' \to K'$
	We again choose a primitive isotropic vector $\tilde{z} \in K$ and $\tilde{z}' \in K'$ with $(\tilde{z}, \tilde{z}') = 1$. 
	For $f \in H_{0,L}^!$ we write
	$$f_K = \sum_{\gamma \in K' / K} \sum_{\substack{\delta \in L_0'/L  \\ p(\delta) = \gamma}} f_{\delta} \frake_\gamma$$
	and
	$$f_{\{0\}} =  \sum_{\gamma \in K'_0 / K} \sum_{\substack{\delta \in L_0'/L  \\ p(\delta) = \gamma}} f_{\delta}.$$
    \begin{lem}
        The isomorphism $p_L \colon L_0'/L_0 \xrightarrow{\cong} K'/K$ identifies $f_K = \downarrow^L_{L_0}(f)$.
        Similarly, $f_{\{0\}} = \downarrow^L_{L_{00}}(f) = \downarrow^L_{\mathrm{type}_L(I)} f$ where $I$ is the isotropic plane generated by $z$ and $\tilde z$.
    \end{lem}
    \begin{proof}
        This follows from \thref{lem:projection-p-isos} and the definition of isotropic descent.
    \end{proof}
	
    Using Lemma \ref{lem:HarmonicMassFormDecomposition} it suffices to consider invariant vectors and Hejhal Poincar\'e series independently.

    \begin{thm}[{\cite[Section 2, 3]{BruinierHabil}}]\label{thm:BorcherdsProductsHejhalPoincareSeries}
        Let $W$ be a Weyl chamber that contains $z$ in its closure. Let $f \in H_{0, L}$ be a harmonic weak Maass form which is a linear combination of Hejhal Poincar\'e series. Then $\Phi(Z, f)$ is given by
        \begin{align*}
            &8 \pi (\rho(K, W, f_K), Y) + C_f - c^+(0, 0) \log(Y^2)\\
            &- 4 \sum_{\substack{\lambda \in K' \\ (\lambda, W) > 0}} \sum_{\substack{\delta \in L_0' / L \\ \pi(\delta) = \lambda + K}} c^+(\delta, q(\lambda)) \log\lvert 1 - e((\lambda, Z) + (\delta, z'))\rvert \\
            &- 4 \sum_{\substack{\lambda \in K' \\ (\lambda, W) > 0}} \sum_{\substack{\delta \in L_0' / L \\ \pi(\delta) = \lambda + K}} c^-(\delta, -q(\lambda)) \log\lvert 1 - e((\lambda, \widetilde{Z}) + (\delta, z'))\rvert,
        \end{align*}
        where $\widetilde{Z} = (z_1, -\overline{z}_2), C_f$ is a constant depending on $f$ and
        $$\rho(K, W, f_K) = \frac{1}{2} \sum_{\substack{b_1 \in \IZ / N_z \IZ \\ b_2 \in \IZ / N_{\tilde{z}} \IZ}} c^+\left(\frac{b_1 z}{N_z} + \frac{b_2 \tilde z}{N_{\tilde{z}}}, 0\right) B_2(b_2 / N_{\tilde{z}}) \tilde{z}$$
        is the so called Weyl vector.
    \end{thm}

    \begin{proof}
        This follows from \cite[Section 2, 3]{BruinierHabil} using that the special function $\calV_2(A, B)$ occurring in the Fourier expansion is given by
        $$\calV_{2}(A, B) = \sqrt{\pi} (2 \cdot 2 \sqrt{A^2 + B^2} / \pi)^{1/2} K_{-1/2}(2 \sqrt{A^2 + B^2}) = \sqrt{\pi} e^{-2 \sqrt{A^2 + B^2}}.$$
        By \cite[Section 3.2]{BruinierHabil} using
        $$\sqrt{\lvert \lambda \rvert^2 Y^2 + (\lambda, Y)^2} = \lvert (\lambda, Y) \rvert$$
        we obtain the Fourier expansion
        \begin{align*}
            &\frac{\lvert Y \rvert}{\sqrt{2}} \Phi^{K}(Y / \lvert Y \rvert, f) + C_f - c^+(0, 0) \log(Y^2)\\
            &- 4 \sum_{\substack{\lambda \in K' \\ (\lambda, W) > 0}} \sum_{\substack{\delta \in L_0' / L \\ \pi(\delta) = \lambda + K}} c^+(\delta, q(\lambda)) \log\lvert 1 - e((\lambda, Z) + (\delta, z'))\rvert \\
            &+ 2 \sum_{\substack{\lambda \in K' \\ q(\lambda) < 0}} \sum_{\substack{\delta \in L_0' / L \\ \pi(\delta) = \lambda + K}} c^-(\delta, q(\lambda)) \sum_{n = 1}^\infty \frac{1}{n} e(ni\sqrt{\lvert \lambda \rvert^2 Y^2 + (\lambda, Y)^2} + n(\lambda, X) + n(\delta, z')) \\
            &= \frac{\lvert Y \rvert}{\sqrt{2}} \Phi^{K}(Y / \lvert Y \rvert, f) + C_f - c^+(0, 0) \log(Y^2)\\
            &- 4 \sum_{\substack{\lambda \in K' \\ (\lambda, W) > 0}} \sum_{\substack{\delta \in L_0' / L \\ \pi(\delta) = \lambda + K}} c^+(\delta, q(\lambda)) \log\lvert 1 - e((\lambda, Z) + (\delta, z'))\rvert \\
            &- 4 \sum_{\substack{\lambda \in K' \\ q(\lambda) > 0 \\ (\lambda, W) > 0}} \sum_{\substack{\delta \in L_0' / L \\ \pi(\delta) = \lambda + K}} c^-(\delta, -q(\lambda)) \log \lvert 1 - e((\lambda, \widetilde{Z}) + (\delta, z')) \rvert.
        \end{align*}
        Moreover, the results of \cite[Section 3.1, 2.3]{BruinierHabil} yield
        \begin{align*}
            &\frac{\lvert Y \rvert}{\sqrt{2}} \Phi^{K}(Y / \lvert Y \rvert, f) \\
            &= 4 \pi y_2 \sum_{\substack{b_1 \in \IZ / N_z \IZ \\ b_2 \in \IZ / N_{\tilde{z}} \IZ}} c^+\left(\frac{b_1 z}{N_z} + \frac{b_2 \tilde z}{N_{\tilde{z}}}, 0\right) B_2(b_2 / N_{\tilde{z}}). \qedhere
        \end{align*}
    \end{proof}
 
	\begin{thm}[{\cite[Theorem 13.3]{Borcherds}}]\thlabel{thm:BorcherdsProductsWeaklyHolomorphic}
		Let $W$ be a Weyl chamber whose closure contains $z$ and let $f \in M_{0,L}^{!}$. Then there is a modular form $\Psi_f$ of weight $c(0, 0) / 2$ with respect to $\Gamma(L)$ and some multiplier system $\chi$, such that
        $$\log\lvert \Psi_f(Z) \rvert = -\frac{1}{4}\Phi(Z, f) - \frac{c^+(0, 0)}{4} \log(Y^2) + \frac{C_f}{4},$$
        where $C_f$ is a constant depending on $f$. The zeros and poles of $\Psi_f$ lie on rational quadratic divisors $\lambda^\perp$ for $\lambda \in L$ with $q(\lambda) < 0$ and have order
		$$\sum_{x \lambda \in L} c(x \lambda, q(x \lambda)).$$
		Moreover, for each primitive norm $0$ vector $z \in L$, the function $\Psi_f$ has an infinite product expansion in a neighborhood of the cusp $z$ given by 
  \begin{align*}
      \Psi_f(z_1, z_2) 
      &= C e(\varrho(K, W, f_K), z_1 \tilde{z} + z_2 \tilde{z}') \\
      &\times \prod_{\substack{\lambda_1, \lambda_2 \in \IZ_{\geq 0} \\ (\lambda_1, \lambda_2) \neq (0, 0)}} \prod_{\substack{b \in \IZ / N_z \IZ}} (1 - e(\lambda_1 z_2 + \lambda_2 z_1 + b / N_z))^{c(\lambda_1 \tilde{z} / N_{\tilde{z}} + \lambda_2 \tilde{z} / N_{\tilde{z}} + bz / N_z, \lambda_1 \lambda_2 / N_{\tilde{z}})},
  \end{align*}
        where
        $$\varrho(K, W, f_K) = \varrho_{\tilde{z}} \tilde{z} + \varrho_{\tilde{z}'} \tilde{z}'$$
        with $\varrho_{\tilde{z}'}$ being the constant term of $f_{\{0\}} E_2(\tau) / 24$ and $\varrho_{\tilde{z}}$ can be determined as in \cite[Section 10]{Borcherds}.
	\end{thm}
	
	When the input $f$ of the Borcherds lift is holomorphic, i.e. $f  = \frakv_f \in \IC[L'/L]^{\SL_2(\IZ)}$ is an invariant vector, the divisor of $\Psi_f$ on $Y(L)$ vanishes. Since the boundary has codimension one in our setting, $\Psi_f$ can still have a non-vanishing divisor along the boundary.  
	For $L = L_{N,N'}$ such Borcherds products are explicitly computed in \cite[Section 5]{Bieker}.
    We need the following two results.
	Let
    $$\eta : \IH \to \IC, \qquad \eta(\tau) = e(\tau / 24) \prod_{n = 1}^\infty (1 - e(n \tau))$$
    be the \emph{Dedekind eta function}.
    \begin{prop}[{\cite[Proposition 5.1]{Bieker}}]
        \thlabel{propModUnits-NN'}
        Let $L = L_{N,N'}$ and $\frakv \in \IQ[L'/L]^{\SL_2(\IZ)}$.
        Let $\frakv = v^{H^*_{a,c}}$ be the characteristic function of the self-dual isotropic subgroup $H^*_{a,c}$ and let $\Psi_{a,c}^* = \Psi_\frakv$ be the corresponding Borcherds product. Then $\Psi_{a,c}^{(2)}(z_1, z_2) = \Psi_{a,c}^{(1)}(z_2, z_1)$ and the product expansion of $\Psi^{(1)}_{a,c}(z_1,z_2)$ in the cusp $(\infty, \infty)$ is up to a non-zero constant given by 
            $$\eta \left( \frac{\frac{Nd_1}{N'M} z_1 + u d_2}{N'/d_1} \right) \cdot \eta \left( \frac{N}{N'M} z_2 \right),$$
        where $d_1 = \gcd(a, N')$, $d_2 = \gcd(c/M, N')$, $M = \gcd(N / N', c)$ and $u \in (\IZ/N'\IZ)^\times$ are as in \thref{lem:isotropic-subgroups-cusps}.
    \end{prop}
    In the two cases $N' = 1$ and $N'=N$ the formulas simplify as follows.
	\begin{cor}
		\thlabel{propModUnits}
		\begin{enumerate}
			\item 
            \label{propModUnits-N1}
            Let $N'= 1$. Then the product expansion of $\Psi^{(1)}_{a,c}= \Psi^{(2)}_{a,c}$ in the cusp $(\infty, \infty)$ is up to a non-zero constant given by
			$$ \eta(dz_1) \eta(d z_2), $$ 
			where $d = \frac{N}{N'M}$ and $M = \gcd(N / N', c)$.
			\item 
            \label{propModUnits-NN}
            Let $N = N'$. Then the product expansion of $\Psi^{(1)}_{a,c}$ in the cusp $(\infty, \infty)$ is up to a non-zero constant given by 
            $$\eta \left( \frac{d_1 z_1 + u d_2}{N/d_1} \right) \cdot \eta \left( z_2 \right).$$
		\end{enumerate}
	\end{cor}

    \begin{rem}\thlabel{rem:EtaProductIdentity}
        The linear relation we found in \thref{lem:space-types-dimension} lifts to the identity of $\eta$-quotients
        $$ \prod_{u \in (\IZ/p^r\IZ)^\times} \eta\left( \frac{z + u}{p^r} \right) = \exp \left( \frac{p^{r} - p^{r-1}}{48} \right) \frac{\eta(z)^{p^r-p^{r-1}+2}}{\eta(z/p^r) \eta(p^rz)}$$
        for every prime $p$ and every $r \geq 1$ generalizing the result \cite[Corollary 5.6]{Bieker} to arbitrary $r$.
    \end{rem}

	\section{Boundary divisors of Borcherds products}

	We determine the boundary divisors of Borcherds products.
    In order to compute the multiplicity of a one-dimensional boundary component in the divisor of a Borcherds product we use the notion of types of one-dimensional boundary components introduced in Section \ref{sec:OneDimCusps}.
    As before, let $L$ be an even lattice of signature $(2,2)$ and Witt rank 2.
    \begin{defn}
    \begin{enumerate}
        \item 
        Let $H \subset L'/L$ be a self-dual isotropic subgroup.
        The \emph{special boundary divisor associated to $H$} is defined to be the boundary divisor
        $$ Z(H) = \sum_{S} \#(H \cap \mathrm{type}_L(S)) [S]$$
        in $X(L)$, where the sum runs over all one-dimensional cusps of $X(L)$.
        \item 
        A divisor on $X(L)$ is called \emph{special} if its restriction to $Y(L)$ is a linear combination of Heegner divisors and if its boundary part is a linear combination of special boundary divisors $Z(H)$ for self-dual isotropic subgroups $H \in \mathrm{Types}(L)$ of $L'/L$.
        The space of special boundary divisors is denoted by $\spbdiv(X(L))_\IQ$.
    \end{enumerate}
    \end{defn}
	We use the following criterion to control divisors of Borcherds products.
	\begin{prop}\thlabel{lem:boundary-mult-weyl-vect}
		Let $f \in H_{0, L}$ such that there exists a meromorphic modular form $F \colon \IH\times \IH \to \IC$ with $- 4\log|F| = \Phi(Z,f)$ (e.g. for $f \in M^!_{0,L}$).
        Let $$f(\tau) = \frakv_f + \sum_{\gamma \in L' / L} \sum_{n < 0} c^+(\gamma, n) F_{\beta, m}(\tau, 1)$$ be the decomposition of $f$ into an invariant vector $\frakv_f$ and a linear combination of Hejhal-Poincar\'e series as in Lemma \ref{lem:HarmonicMassFormDecomposition}. Then the following hold:
        \begin{enumerate}
            \item 
            The divisor of $F$ on $Y(L)$ depends only on $f - \frakv_f$ and is a sum of Heegner divisors.
            \item 
            Let $I$ be a rational isotropic plane in $V = L\otimes_\IZ \IQ$.
		      The order of $F$ along the one-dimensional cusp of $X(L)$ corresponding to $I$ is given by $\varrho_{\tilde{z}'} = \downarrow_{\type_L(I)}^L \frakv_f$.
            In particular,it only depends on $\frakv_f$ and $\type_L(I)$.
            \item  When $f = \frakv^H$ is the characteristic function of a self-dual isotropic subgroup $H \subset L'/L$ the boundary divisor of $F$ is given by $Z(H)$.
        \end{enumerate}
	\end{prop}
	\begin{proof}
 \begin{enumerate}
    \item This is \cite[Theorem 3.16]{BruinierHabil} and \cite[Theorem 13.3]{Borcherds}, compare \thref{thm:BorcherdsProductsWeaklyHolomorphic}.
     \item 
		This can be read off from the product expansions of $F$ as in Theorem \ref{thm:BorcherdsProductsHejhalPoincareSeries} and Theorem \ref{thm:BorcherdsProductsWeaklyHolomorphic}, compare \cite[Corollary 2.3, Remark 2.2]{KudlaProduct}.
    \item 
    By the definition of Weyl vectors in this case we find $\varrho_{\tilde z'} = \downarrow^L_{\type_L(I)} \frakv^H = \#(H \cap \type_L(I))$.\qedhere
 \end{enumerate}
	\end{proof}

    In particular, it suffices to study the boundary divisors of lifts of invariant vectors.
    As the multiplicity of the boundary divisor of a Borcherds product along a one-dimensional cusp only depends on the type of the cusp, we may identify $\spbdiv(X(L))_\IQ$ with a subspace of $\IQ^{\mathrm{Types}(L)}$.
    The map that associates to an invariant vector the boundary divisor of its Borcherds product
    $$ \di_\partial \circ \Psi(Z, -) \colon \ \IQ[L'/L]^{\SL_2(\IZ)} \to \spbdiv(X(L))_\IQ $$ is a linear map of $\IQ$-vector spaces. The surjectivity follows from \thref{lem:boundary-mult-weyl-vect}. In particular, $\spbdiv(X(L))_\IQ$ is exactly the space of boundary divisors of Borcherds products.

    In general the map will fail to be an isomorphism as the number of types of one-dimensional cusps of $L$ is strictly smaller than the dimension of the space of invariants, compare the discussion following \thref{lem:space-types} above.
    However, we show that the restriction of $\di_\partial \circ \Psi(Z, -)$ to $\IQ[L'/L]_{\mathrm{Types}(L)}$ becomes an isomorphism.
    This allows us to describe the space $\spbdiv(X(L))_\IQ$ more explicitly in this case.

    Taking the product (respectively the sum) of isotropic descent  (respectively isotropic induction) for all types of one-dimensional cusps of $L$ we obtain a pair of adjoint operators
    $$ \prod_{H \in \mathrm{Types}(L)} \left(\downarrow^{L' / L}_H - \right) \colon \IC[L' / L]^{\SL_2(\IZ)} \to  \bigoplus_{H \in \mathrm{Types}(L)} \IC[H^{\perp}/H]^{\SL_2(\IZ)} $$ and 
    $$ \bigoplus_{H \in \mathrm{Types}(L)} \left( \uparrow^{L' / L}_H - \right)  \colon \bigoplus_{H \in \mathrm{Types}(L)} \IC[H^{\perp}/H]^{\SL_2(\IZ)} \to \IC[L' / L]^{\SL_2(\IZ)}.$$
    Note that by definition $\IC[L'/L]_{\mathrm{Types}(L)}$ is the image of $\bigoplus_{H \in \mathrm{Types}(L)} \left( \uparrow^{L' / L}_H - \right)$. 
    
    \begin{prop}
        \thlabel{prop:boundary-isotropic-descent}
        The map
        $$\di_\partial \circ \Psi(Z, -) \colon \IQ[L'/L]^{\SL_2(\IZ)} \to \spbdiv(X(L))_\IQ \hookrightarrow \IQ^{\mathrm{Types}(L)}$$
        agrees with the map that associates to an invariant vector the zero-component of its isotropic descent
        $$ \prod_{H \in \mathrm{Types}(L)} (\downarrow^{L' / L}_H - )_0 \colon \IQ[L'/L]^{\SL_2(\IZ)} \to \IQ^{\mathrm{Types}(L)}.$$
        Moreover, the map restricts to an isomorphism
        $$ \IQ[L'/L]_{\mathrm{Types}(L)} \xrightarrow{\cong} \spbdiv(X(L))_\IQ.$$
    \end{prop}
    \begin{proof}
        The first part follows from the explicit descriptions of the two maps in \thref{lem:boundary-mult-weyl-vect} and the definition of the isotropic descent. 
        The second statement follows from the fact that isotropic descent and induction are adjoint together with the fact that $\IQ[L'/L]_{\mathrm{Types}(L)}$ and $\spbdiv(X(L))_\IQ$ are the images of isotropic induction and isotropic descent, respectively.
    \end{proof}

    The following theorem is a reformulation of the previous proposition together with \thref{lem:boundary-mult-weyl-vect} and \thref{lem:space-types-dimension}.
    \begin{thm}\thlabel{thm:SpecialBoundaryDivisorsAsBorcherdsProducts}
        Let $f \in M_{0, L}^!$. Then the boundary divisor of $\Psi_f(Z)$ on $X(L)$ is special and every special boundary divisor is the boundary divisor of $\Psi_\frakv(Z)$ for some $\frakv \in \IC[L' / L]_{\Types(L)}$.
        For $L = L_{N,N'}$ the space of special boundary divisors has dimension 
        $$\dim \spbdiv(X(L))_\IQ = \prod_{p \mid N'} \left(2((r_p-r_p'+1)p^{r_p'} - (r_p-r_p'-1) p^{r_p' - 1}) - 1\right)  \cdot \prod_{p \nmid N'}  ({r_p}+1).$$
    \end{thm}
    
    \begin{lem}
    \thlabel{lem:trivial-divisor-constant}
        Let $F$ be a modular form for $\Gamma(L)$ of some weight $k$ and some character (of finite order) such that the divisor of $F$ vanishes on $X(L)$. Then $F$ is constant.
    \end{lem}

    \begin{proof}
        By replacing $F$ with a suitable power we may assume that $F$ has trivial character.
        Then $\frac{1}{F}$ is also a holomorphic modular form of weight $-k$. As there are no holomorphic modular forms of negative weight, this implies that $F = 0$ when $k \neq 0$. 
        When $k= 0$ this means that $F$ is constant.
    \end{proof}

    \begin{cor}\thlabel{cor:ModularUnitsThatAreBorcherdsProducts}
        Let $F : \IH \times \IH \to \IC$ be a modular unit of some weight and some character of finite order for $\Gamma(L)$. Then $F$ is a Borcherds product up to a constant if and only if it has a special boundary divisor. In particular, such a modular unit is a product of the eta-functions appearing in \thref{propModUnits-NN'} and the constructions of \cite{Bieker} indeed produce all modular units.
    \end{cor}

    \begin{proof}
        Since the boundary divisor of $F$ is special, there is $\frakv \in \IQ[L' / L]^{\SL_2(\IZ)}$ such that $\Psi_\frakv(Z)$ has the same boundary divisor by \thref{thm:SpecialBoundaryDivisorsAsBorcherdsProducts}. Then the divisor of $F / \Psi_\frakv$ vanishes on $X(L)$. By \thref{lem:trivial-divisor-constant} we obtain the result. That $F$ is a product of eta-functions can be seen by choosing a finite index sublattice $L_{N, N} \subseteq L$ using \thref{lem:sublattice-LNN} and observing that $\Gamma(L_{N, N}) \subseteq \Gamma(L)$.
    \end{proof}

	In the two cases $L = L_{N,1}$ and $L = L_{p^r,p^r}$ we more explicitly determine the space of boundary divisors of Borcherds products $\spbdiv(X(L)) \subseteq \Div_\partial(X(L))_\IQ$.

    \begin{cor}
    \thlabel{cor:boundary-div-N1}
		Let $L = L_{N,1}$. 
        A boundary divisor on $X(L_{N,1})$ arises as a divisor of a Borcherds product if and only if it
		\begin{enumerate}
			\item is symmetric and 
			\item its multiplicity along $\{\frac{a}{c}\} \times \IH$ only depends on $\gcd(c,N)$.
		\end{enumerate}        
    \end{cor}

	\begin{cor}\thlabel{cor:dimension-boundary-divisors}
            Let $L = L_{p^r,p^r}$. Then the subspace $ \spbdiv(X(L_{p,p}))_\IQ \subset  \Div_\partial(X(L_{p,p}))_\IQ$ is characterised by the following conditions.
            The multiplicities
            \begin{enumerate} 
            \item for boundary components of the form $\{\frac{a}{c} \} \times \IH$ with $p \mid a$ only depend on $a c^{-1} \in \IZ/p^r\IZ$ and
            \item for boundary components of the form $\{\frac{a}{c} \} \times \IH$ with $\gcd(p,a) = 1$ only depend on $a^{-1} c \in \IZ/p^r\IZ$. 
        \end{enumerate}
        The corresponding conditions are also required for the boundary components of the form $\IH \times \{\frac{a}{c} \}$. 
        Moreover, 
        \begin{enumerate}[resume]
            \item
            \label{it:cond-boundary-pp:sum-multiplicities}
            the total sum of multiplicities of boundary components of the form $\{\frac{a}{c} \} \times \IH$ agrees with the sum of multiplicities of components of the form $\IH \times \{\frac{a}{c} \}$.
        \end{enumerate}
	\end{cor}

   \begin{proof}[Proof of \thref{cor:boundary-div-N1} and \thref{cor:dimension-boundary-divisors}]
       In both cases conditions (1) and (2) say that the multplicity of a one-dimensional cusp only depends on its type. Moreover, \eqref{it:cond-boundary-pp:sum-multiplicities} is the lift of condition \eqref{lem:space-types-relation-eq}.
   \end{proof}

    \begin{rem}
        One can also determine the space of boundary divisors for $L = L_{N,N'}$ more directly by computing the boundary divisors of the $\eta$-products of \thref{propModUnits-NN'} using for example \cite[Proposition 2.1]{Koehler} and \cite[Proposition 6.2]{ScheithauerWeilRep}.
    \end{rem}

    \begin{rem} 
        By \thref{propModUnits-NN'}, the Borcherds lift $\Psi_{\frakv}$ for an invariant vector $\frakv \in \IC[L'/L]^{\SL_2(\IZ)}$ for $L = L_{N,N'}$ factors as $\Psi_\frakv(z_1, z_2) = \psi_1(z_1) \psi_2(z_2)$. Both components $\psi_1, \psi_2$ are modular forms for 
        $$\Gamma = \left\{\begin{pmatrix}
            a & b \\
            c & d
        \end{pmatrix} \in \SL_2(\IZ) \ \bigg \vert \ N \mid c, \ N' \mid b,  \ a \equiv d \equiv 1 \!\pmod{N'} \right\}.$$
        In particular, for $N' = 1$ we have $\Gamma = \Gamma_0(N)$ and for $N' = N$ we have $\Gamma = \Gamma(N)$.   
        Moreover, the divisors of the $\psi_i$ are supported on the cusps.
        Modular functions with divisors supported on the cusps are usually called modular units following \cite{KubertLangBook}.

        By \cite[Theorem 5]{Newman}, all modular units for $\Gamma_0(N)$ are $\eta$-products and appear as components of Borcherds lifts in \thref{propModUnits} (\ref{propModUnits-N1}).
        However, by \cite[Theorem 2.3]{KubertLang2} the space of boundary divisors generated by boundary divisors of modular units for $\Gamma(N)$ agrees with the full space of cuspidal divisors, a set of generators is given by Klein forms.
		For $N \geq 5$ the subspace of boundary divisors of $\eta$-quotients as in \thref{propModUnits} (\ref{propModUnits-NN}) is a proper subspace as the dimension is strictly smaller than the number of cusps. 
		This shows in particular that not all modular units for $\Gamma(N)$ arise as components of Borcherds products for $\Gamma(L_{N,N})$ in this case.
	\end{rem}

\renewcommand\bibname{References}
\bibliographystyle{alphadin}
\bibliography{./bibliography}

\end{document}